\documentclass[12pt,a4paper]{amsart}
\usepackage{fancyhdr}
\usepackage{appendix}
\usepackage{amssymb,amscd,amsxtra,calc}
\usepackage{mathrsfs}
\usepackage{amsmath}
\usepackage{multirow}
\usepackage{verbatim}
\usepackage[all]{xy}
\usepackage[colorlinks,linkcolor=red,anchorcolor=blue,citecolor=blue]{hyperref}
\theoremstyle{plain}
    \newtheorem{thm}{Theorem}[section]

    \newtheorem{corollary}[thm]{Corollary}
    
    \newtheorem{lemma}[thm]{Lemma}
    \newtheorem{proposition}[thm]{Proposition}
    \newtheorem{question}[thm]{Question}
    \newtheorem{theorem}[thm]{Theorem}

\theoremstyle{definition}

    \newtheorem*{notation*}{Notation and Terminology}
    \newtheorem{remark}[thm]{Remark}

\theoremstyle{remark}

\makeatletter

\newcommand{\Rmnum}[1]{\expandafter\@slowromancap\romannumeral #1@}
\makeatother

\begin{document}

\title[$Q$-complex tori]
{Characterization of $Q$-complex tori via endomorphisms -- an addendum to ``Int-amplified endomorphisms of compact K\"ahler spaces''}

\author{Guolei Zhong}
\address
{
\textsc{Department of Mathematics,
	National University of Singapore,
	Singapore 119076, Republic of Singapore.
\newline\indent 
Current address: Center for Complex Geometry, Institute for Basic Science (IBS), 55 Expo-ro,  Yuseong-gu, Daejeon, 34126, Republic of Korea.
}}
\email{zhongguolei@u.nus.edu, guolei@ibs.re.kr}

\begin{abstract}
In this short note, we consider a normal compact K\"ahler klt space $X$ whose canonical divisor $K_X$ is pseudo-effective, and give a dynamical criterion for $X$ to be a $Q$-complex torus. 
We show that, if such $X$ admits an int-amplified endomorphism, then $X$ is a $Q$-complex torus. 
As an application, we prove that, if a simply connected compact K\"ahler (smooth) threefold admits an int-amplified endomorphism, then it is (projective and) rationally connected.
\end{abstract}


\maketitle


\section{Introduction}
We work over the field $\mathbb{C}$ of complex numbers. 
Let $(X,\omega)$ be a compact K\"ahler manifold of dimension $n$ such that $c_1(X)=0$ and $\int_Xc_2(X)\wedge\omega^{n-2}=0$.
Following from Yau's solution of Calabi conjecture \cite{Yau78}, the vanishing of the first Chern class $c_1(X)$ implies that we can find a Ricci flat metric $\omega$.
On the other hand, the vanishing $\int_Xc_2(X)\wedge \omega^{n-2}=0$ implies that the full curvature tensor of $\omega$ is identically zero.
By the uniformization theorem, the universal cover of $X$ is an affine space and $X$ is a quotient of a complex torus $T$ by a finite group acting freely on $T$. 

From the viewpoints of birational geometry, 
since the end product of the minimal model program (MMP for short)  is usually singular, it is important to consider the singular version of the uniformization theorem. 
We say that a compact K\"ahler space $X$ is a \textit{$Q$-complex torus}, if there is a finite morphism $T\to X$ from a complex torus $T$, which is \'etale in codimension one (cf.~\cite[Definition 3.6]{Nak99}, \cite[Definition 2.13]{NZ10} and \cite[Section 5]{Zho21}). 
Thanks to many people's efforts, when $X$ is projective, the numerical characterization of $Q$-complex tori has been successfully generalized to singular  varieties (see~\cite{GKP16} and \cite{LT18}). 
In the transcendental case however, the slicing arguments used to reduce a projective variety to a complete intersection surface (by integral ample divisors) do not work any more.
Recently, Kirschner and Graf settled this problem in \cite{GK20} for K\"ahler threefolds, and in higher dimensions, Claudon, Graf and Guenancia settled the uniformization problem in  \cite{CGG23} (cf.~\cite{CGG22}). 

From the dynamical viewpoints, one of the application of the numerical characterization of $Q$-complex tori is to describe the end product of the equivariant minimal model program (E-MMP for short). 
In other words, given a holomorphic self-map $f$ on a compact K\"ahler space $X$ with mild singularities, we are interested in the construction of the $f$-equivariant MMP for  $X$, assuming the existence of an MMP for $X$.  
We refer the reader to our recent works \cite{Zho22} and \cite{Zho23} for the aspect of automorphism groups.

When $X$ is projective, Nakayama first established the E-MMP for surfaces admitting non-isomorphic surjective endomorphisms (see \cite{Nak20a},  \cite{Nak20b}). 
In higher dimensions, Meng and Zhang constructed in \cite{MZ18} and \cite{Men20} the $f$-E-MMP when $f$ is \textit{polarized}, i.e., $f^*H\sim qH$ for some (integral) ample divisor $H$ and integer $q>1$,  
or more generally, $f$ is \textit{int-amplified}, i.e., $f^*H-H=L$ for some (integral) ample divisors $H$ and $L$. 
In particular, the following Theorem \ref{thm-proj-Q-torus}  characterizes numerically the end product of the E-MMP (in the projective case)  by considering the first Chern class and the orbifold second Chern class.
\begin{theorem}[{\cite[Theorem 1.21]{GKP16}, \cite[Theorem 5.2]{Men20}; cf.~\cite[Theorem 3.4]{NZ10}}]\label{thm-proj-Q-torus}
Let $X$ be a normal projective klt variety admitting an int-amplified endomorphism $f$.
Suppose the canonical divisor $K_X$ is pseudo-effective.
Then $X$ is a $Q$-complex torus.
\end{theorem}

In \cite{Zho21}, the author generalized the ideas of the E-MMP for projective varieties to the transcendental case, and constructed the E-MMP for terminal threefolds. 
In the higher dimensional case however, due to the lack of the existence of an MMP, the E-MMP hasn't been established in its full version so far.  
Recall that a holomorphic self-map $f:X\to X$ on a compact K\"ahler space $X$ with only rational singularities is said to be \textit{int-amplified}, if the induced linear operation $f^*|_{\textup{H}^{1,1}_{\textup{BC}}(X)}$ on the Bott-Chern cohomology space $\textup{H}^{1,1}_{\textup{BC}}(X)$ (see~\cite[Definition 3.1]{HP16}) has all the eigenvalues being of modulus greater than 1, or equivalently, $f^*[\omega]-[\omega]=[\eta]$ for some K\"ahler classes $[\omega]$ and $[\eta]$ (cf.~\cite[Theorem 1.1]{Zho21}). 
We refer the reader to \cite[Section 2]{Zho21} for notations and terminologies involved. 

Based on this generalized notion, it is natural for us to ask the following question, as an analytic version of Theorem \ref{thm-proj-Q-torus}.
\begin{question}\label{ques-q-torus}
Let $X$ be a normal compact K\"ahler space with at worst rational singularities.
Suppose that $X$ admits an int-amplified endomorphism and the canonical divisor $K_X$ is pseudo-effective.
Is $X$  a $Q$-complex torus?	
\end{question}
We refer the reader to \cite[Theorem 1.4]{Zho21} (cf.~\cite[Section 6]{Zho21}) for a positive answer to Question \ref{ques-q-torus} when $X$ satisfies one of the following: (1) $X$ is smooth; (2) $\dim X\le 2$;  or (3) $X$ is a terminal threefold. 

In this short note, we study Question \ref{ques-q-torus} and    characterize the $Q$-complex tori dynamically in a more general setting. 

\begin{remark}
\begin{enumerate}
    \item After the previous paper \cite{Zho21} being accepted, the  characterization of $Q$-complex tori has been greatly extended from threefolds \cite{GK20} to higher dimensions \cite{CGG22, CGG23}. 
Hence, together with our dynamical viewpoints, one of the main goals of this paper is to prove a higher dimensional analogue of \cite[Theorem 1.4]{Zho21}.
\item In the papers \cite{Zho21} and \cite{Zho22}, the notion ``$Q$-torus'' therein should be understood as ``$Q$-complex torus'' (cf.~\cite[Definition 3.6]{Nak99}).
\end{enumerate}
\end{remark}

Theorems \ref{thm-Q-torus} \(\sim\) \ref{simply-connected} below are our main results.

\begin{theorem}\label{thm-Q-torus}
Let $X$ be an \(n\)-dimensional compact K\"ahler klt space whose canonical divisor \(K_X\) is pseudo-effective. 
Suppose that $X$ admits an int-amplified endomorphism $f$. 
Then $X$ is a $Q$-complex torus. 
In particular, Question \ref{ques-q-torus} has a positive answer if \(X\) has only klt singularities.
\end{theorem}

There are  different kinds of applications of Theorem \ref{thm-Q-torus}.
On the one hand,  we can characterize the end product of the  E-MMP   for compact K\"ahler klt spaces in arbitrary dimension (as a quasi-\'etale quotient of a complex torus), if the MMP exists. 
On the other hand, since the existence of the MMP has been nicely confirmed for K\"ahler klt threefolds (\cite{DH22} and \cite{DO22}; cf.~\cite{HP16}, \cite{HP15}, \cite{CHP16} and \cite{DO23}),  we can reformulate the E-MMP for klt  threefolds as below, extending the previous E-MMP for terminal threefolds.

\begin{theorem}[{cf.~\cite[Theorem 1.5]{Zho21}}]\label{E-MMP-threefolds}
Let $f:X\to X$ be an int-amplified endomorphism 	of a compact K\"ahler klt threefold. 
Then, after replacing $f$ by a power, there exists an $f$-equivariant minimal model program
$$X=X_1\dashrightarrow\cdots\dashrightarrow X_i\dashrightarrow\cdots\dashrightarrow X_r=Y$$
(i.e., $f=f_1$ descends to each $f_i$ on $X_i$) with each $X_i\dashrightarrow X_{i+1}$ a divisorial contraction, a flip or a Fano contraction, of a $K_{X_i}$-negative extremal ray. 
Moreover, 
\begin{enumerate}
\item If $K_X$ is pseudo-effective, then $X=Y$ is a $Q$-complex torus.
\item If $K_X$ is not pseudo-effective, then for each $i$, $f_i$ is int-amplified and the composite $X_i\to Y$ is equidimensional and holomorphic with every fibre being reduced and irreducible.
Furthermore, every fibre of $X_i\to Y$ is rationally connected; and the last step $X_{r-1}\to X_r$ is a Fano contraction.
\end{enumerate}
\end{theorem}

Finally, running the E-MMP for a compact K\"ahler (smooth) threefold $X$, we obtain the following theorem,  which is a transcendental version of \cite[Corollary 1.4]{Yos21} in low dimension.
In this case when  $X$ is simply connected, $X$ admitting an int-amplified endomorphism forces it to be projective.

\begin{corollary}[{cf.~\cite[Corollary 1.4]{Yos21}}]\label{simply-connected}
Let $X$ be a  simply connected compact K\"ahler (smooth) threefold admitting an int-amplified endomorphism. 
Then $X$ is a (projective) rationally connected threefold. 
In particular, it is of Fano type, i.e., there exists an effective Weil $\mathbb{Q}$-divisor $\Delta$ on $X$ such that $(X,\Delta)$ is klt and $-(K_X+\Delta)$ is ample.	
\end{corollary}

\subsubsection*{\textbf{\textup{Acknowledgments}}}
The author would like to thank Professor De-Qi Zhang for many valuable discussions. 
The author would also like to thank the referees for the very careful reading, pointing out the references, and many constructive suggestions to improve the paper.  
The author was partially supported by a Graduate Scholarship from the department of Mathematics in NUS and was partially supported  by the Institute for Basic Science (IBS-R032-D1-2023-a00).

\section{Proofs of the statements}
\subsection{Characterization of Q-complex tori}
In this subsection, we prove Theorem \ref{thm-Q-torus}. 
We refer the reader to \cite[Section 5]{GK20} for the definition of the first Chern class and the second Chern class in the singular K\"ahler setting.
For the convenience of readers and us, we briefly recall them here.

Let $X$ be a normal complex space of dimension $n$ such that the canonical divisor $K_X$ is $\mathbb{Q}$-Cartier, i.e., there exists some integer $m>0$ such that the reflexive tensor power $\omega_X^{[m]}:=(\omega_X^{\otimes m})^{\vee\vee}=\mathcal{O}_X(mK_X)$ is locally free, where $\omega_X$ is the dualizing sheaf.
The \textit{first Chern class} $c_1(X)$ of $X$ is then defined as
$$c_1(X):=-\frac{1}{m}\mathcal{O}_X(mK_X)\in \textup{H}^2(X,\mathbb{R}).$$
In the following, we further assume that $X$ has only klt singularities.
Let $X^{\circ}\subseteq X$ be the open locus of quotient singularities of $X$; then the \textit{second orbifold Chern class} of $X$ is defined as the unique element in the Poincar\'e dual space $\widetilde{c_2}(X)\in \textup{H}^{2n-4}(X,\mathbb{R})^{\vee}$ whose restriction to $\textup{H}^{2n-4}(X^\circ,\mathbb{R})^\vee$ is the usual second orbifold Chern class $\widetilde{c_2}(X^\circ)\in \textup{H}^4(X^\circ,\mathbb{R})$ (\cite[Definition 5.2]{GK20}).

Before we prove  Theorem \ref{thm-Q-torus}, let us recall the weak numerical equivalence defined in \cite[Section 3]{Zho21}.
Let $f:X\rightarrow X$ be a  surjective endomorphism of a normal compact K\"ahler space $X$ of dimension $n$ with at worst rational singularities. Denote by \setlength\parskip{0pt}
$$\textup{L}^k(X):=\{\sum [\alpha_1]\cup\cdots\cup[\alpha_k]~|~[\alpha_i]\in \textup{H}_{\text{BC}}^{1,1}(X)\},$$
the subspace of $\textup{H}^{2k}(X,\mathbb{R})$. 
Here, $\textup{H}^{1,1}_{\textup{BC}}(X)$ denotes the Bott-Chern cohomology space (see \cite[Definition 3.1]{HP16}).
Let $\textup{N}^k(X):=\textup{L}^k(X)/\equiv_w$, where a cycle $[\alpha]\in \textup{L}^k(X)$ is weakly numerically equivalent (denoted by $\equiv_w$) to zero if and only if for any $[\beta_{k+1}],\cdots,[\beta_n]\in \textup{H}_{\text{BC}}^{1,1}(X)$, $[\alpha]\cup[\beta_{k+1}]\cup\cdots\cup [\beta_n]=0$. 
Moreover, for any $[\alpha],[\beta]\in \textup{L}^k(X)$, $[\alpha]\equiv_w [\beta]$ if and only if $[\alpha]-[\beta]\equiv_w 0$.
We note that, when $X$ has only rational singularities, there is a natural injection $\textup{H}^{1,1}_{\textup{BC}}(X)\hookrightarrow\textup{H}^2(X,\mathbb{R})$ (cf.~\cite[Remark 3.7]{HP16}); hence the above intersection product makes sense.

Let  $\textup{L}_{\mathbb{C}}^k(X)=\textup{L}^k(X)\otimes_{\mathbb{R}}\mathbb{C}$ and $\textup{N}_{\mathbb{C}}^k(X)=\textup{N}^k(X)\otimes_{\mathbb{R}}\mathbb{C}$. As the pull-back operation $f^*$ on Bott-Chern cohomology space induces a linear operation on  the subspace $\textup{L}^k_{\mathbb{C}}(X)\subseteq H^{2k}(X,\mathbb{C})$ for each $k$, it follows from \cite[Proposition 2.8]{Zho21} that $f^*$ also gives a well-defined linear operation on the quotient space $\textup{N}_{\mathbb{C}}^k(X)$.  
The following lemma was proved in \cite[Lemma 3.4]{Zho21} and we recall it here for our later proof.
\begin{lemma}[{\cite[Lemma 3.4]{Zho21}}]\label{lem-tends-zero}
Let $f:X\to X$ be an int-amplified endomorphism of a normal compact K\"ahler space of dimension $n$.
Suppose that $X$ has at worst rational singularities.
Then, for each $0<k<n$,  all the eigenvalues of $f^*|_{\textup{N}^k_{\mathbb{C}}(X)}$ are of modulus less than $\deg f$. In particular, $\lim\limits_{i\rightarrow+\infty}\frac{(f^i)^*[x]}{(\deg f)^i}\equiv_w 0$ for any $[x]\in \textup{L}^k_{\mathbb{C}}(X)$. 
\end{lemma}

\begin{proof}[Proof of Theorem \ref{thm-Q-torus}]
To prove Theorem \ref{thm-Q-torus}, we resort to \cite[Corollary 1.7]{CGG23} by showing that $c_1(X)=0\in \textup{H}^2(X,\mathbb{R})$ and $\widetilde{c_2}(X)\cdot[\omega]^{n-2}=0$ for some K\"ahler class $[\omega]\in \textup{H}^2(X,\mathbb{R})$. 
Here, $c_1(X)$ is the first Chern class and $\widetilde{c_2}(X)$ is the  second orbifold  Chern class of $X$ as defined above. 
Since $-K_X$ is pseudo-effective by \cite[Theorem 1.3]{Zho21} and $K_X$ is pseudo-effective by assumption, the vanishing of the first Chern class  $c_1(X)=0$ follows. 
Then $K_X=f^*K_X$ and it follows from the purity of  branch loci (cf.~\cite{GR55}) that  
our $f$ is \'etale in codimension one. 
Let $d:=\deg f>1$. 
Fix a K\"ahler class $[\omega]$ on $X$. 
Since $f$ is \'etale in codimension one, by \cite[Proposition 5.6]{GK20},
we get $\widetilde{c_2}(X)\cdot (f^*[\omega])^{n-2}=d\cdot \widetilde{c_2}(X)\cdot [\omega]^{n-2}$.
Then with the equality divided by $d$, we have
\begin{align}\label{eqa-def-xi}
\widetilde{c_2}(X)\cdot [\omega]^{n-2}=\widetilde{c_2}(X)\cdot \frac{f^*[\omega]^{n-2}}{d}=\widetilde{c_2}(X)\cdot\frac{(f^i)^*[\omega]^{n-2}}{d^i}=:\widetilde{c_2}(X)\cdot [x_i], 	
\end{align}
where $\lim\limits_{i\to\infty}[x_i]\equiv_w0$ by Lemma \ref{lem-tends-zero}.
Hence, 
$$(\lim\limits_{i\to\infty}[x_i])\cdot [\omega]^2=\lim_{i\to\infty}[x_i]\cdot[\omega]^2=0.$$
Let \(C\subseteq H^{2n-4}(X,\mathbb{R})\) be the closure of the convex cone  generated by \(\xi_1\cup\cdots\cup\xi_{n-2}\) where \(\xi_i\) are K\"ahler classes.
Let \(C^{\vee}\) be the dual cone of \(C\) in \(H^{2n-4}(X,\mathbb{R})^\vee\).
It is clear that \([\omega]^2\) lies in the interior of \(C^\vee\) and hence \(a[\omega]^2-\widetilde{c_2}(X)\) lies in the interior of \(C^\vee\) for sufficiently large \(a\gg 1\).
In particular, we have 
\[
0\le \widetilde{c_2}(X)\cdot [\omega]^{n-2}=\lim_{i\to\infty}(\widetilde{c_2}(X)-a[\omega]^2)\cdot [x_i]\le 0,
\]
where the first inequality is due to the singular Miyaoka-Yau inequality (see \cite[Theorem 1.6]{CGG23}), the second equality follows from Lemma \ref{lem-tends-zero} and the last inequality is by the construction of the dual cone.
Therefore, we conclude our proof by applying \cite[Corollary 1.7]{CGG23}.
\end{proof}

\subsection{E-MMP for threefolds: the klt case}\label{sec-E-MMP}
In this subsection, we construct the equivariant minimal model program (E-MMP) for int-amplified endomorphism on compact K\"ahler klt threefolds.
To prove Theorem \ref{E-MMP-threefolds}, we refer the reader to  \cite[Section 8]{Zho21} and \cite[Section 8]{Men20} for the technical proofs involved.

Let $X$ be an $n$-dimensional compact K\"ahler space with only rational singularities. 
Let $\textup{N}_1(X)$ be the vector space of real closed currents of bi-degree $(n-1,n-1)$ (or equivalently bi-dimension $(1,1)$) modulo the  equivalence: $T_1\equiv T_2$ if and only if $T_1(\eta)=T_2(\eta)$ for all real closed $(1,1)$-forms $\eta$ with local potentials (cf.~\cite[Definition 3.8]{HP16}).

The following lemma was first proved in \cite{Zha10} when $X$ is projective and reformulated in \cite{Zho21} for terminal threefolds.
Due to the organization of the published version \cite{Zho21}, we skipped its proof therein. 
In this note, we follow the idea of \cite{Zha10} to reprove it in our transcendental case here for the convenience of readers, noting that we also weaken the condition on the singularities herein.
\begin{lemma}[{cf.~\cite[Lemma 8.3]{Zho21}, \cite[Lemma 6.2]{MZ18}}]\label{lemequivariant}
Let $f$ be a surjective endomorphism of a compact K\"ahler klt threefold  $X$ and $\pi:X\rightarrow Y$ a contraction  of a $K_X$-negative extremal ray $R_\Gamma:=\mathbb{R}_{\ge 0}[\Gamma]$ generated by a positive closed $(2,2)$-current $\Gamma$. 
Suppose further that $E\subseteq X$ is an analytic subvariety such that $\dim(\pi(E))<\dim E$ and $f^{-1}(E)=E$. 
Then, up to replacing $f$ by its power, $f(R_{\Gamma})=R_{\Gamma}$ (hence, for any curve $C$ with $[C]\in R_{\Gamma}$, its image $[f(C)]$ still lies in $R_{\Gamma}$), i.e., the contraction $\pi$ is $f$-equivariant.
\end{lemma}
\begin{proof}
Since $\dim \pi(E)<\dim E$, we may fix an irreducible curve $C$ with its current of integration $[C]\in R_{\Gamma}$ such that $C\subseteq E$. 
Then, any other contracted curve is a multiple of $C$ in the sense of currents of integration. Let $\textup{L}_{\mathbb{C}}^1(X):=\textup{H}_{\textup{BC}}^{1,1}(X)\otimes_{\mathbb{R}}\mathbb{C}$ and $\textup{L}_{\mathbb{C}}^1(Y):=\textup{H}_{\textup{BC}}^{1,1}(Y)\otimes_{\mathbb{R}}\mathbb{C}$. 
By \cite{Nak87}, our $Y$ also has at worst rational singularities. 
Let us recall the following exact sequence  in \cite[Proposition 8.1]{HP16},
\begin{align}\label{exact}
0\rightarrow \textup{L}^1_{\mathbb{C}}(Y)\xrightarrow{\pi^*} \textup{L}^1_{\mathbb{C}}(X)\xrightarrow{[\alpha]\mapsto \alpha\cdot[C]}\mathbb{C}\rightarrow 0.
\end{align}
Then, the pull-back $\pi^* \textup{L}^1_{\mathbb{C}}(Y)\subseteq \textup{L}_{\mathbb{C}}^1(X)$ is a linear subspace of codimension $1$.

For each $\xi\in \textup{L}^1_{\mathbb{C}}(X)$, denote by $\xi|_E:=i^*\xi\in \textup{L}^1_{\mathbb{C}}(X)|_E$  the restriction, where $i:E\hookrightarrow X$ is the natural inclusion. 
Consider the linear space $\textup{L}^1_{\mathbb{C}}(X)|_E$ and its subspace
$$H:=(\pi^*\textup{L}_{\mathbb{C}}^1(Y))|_E=i^*\pi^*\textup{L}_{\mathbb{C}}^1(Y)\subseteq i^*\textup{L}_{\mathbb{C}}^1(X)=\textup{L}_{\mathbb{C}}^1(X)|_E.$$
Therefore, $H$ is a linear subspace of $\textup{L}_{\mathbb{C}}^1(X)|_E$ of codimension $0$ or $1$. 
Fixing a K\"ahler class $\omega$ on $X$,  we have
$i^*\omega\cdot C>0$, since the restriction $\omega|_C$ is also a K\"ahler class on the curve $C$ (cf.~\cite[Proposition 3.5]{GK20} or \cite[Proposition 2.6]{Zho21}).

We claim that $H\subseteq \textup{L}_{\mathbb{C}}^1(X)|_E$ is a linear subspace of codimension 1.
Suppose the contrary that $H=\textup{L}_{\mathbb{C}}^1(X)|_E$. 
Then $i^*\omega\in \textup{L}_{\mathbb{C}}^1(X)|_E=H(=(\pi^*\textup{L}_{\mathbb{C}}^1(Y))|_E)$ and thus there exists some $\eta\in \textup{L}_{\mathbb{C}}^1(Y)$ such that $i^*\omega=i^*\pi^*\eta\in H$. 
By projection formula, $0<\deg \omega|_C=i^*\omega\cdot C=\eta\cdot \pi_*i_*C=0$, noting that $C$ is contracted by $\pi$, which is absurd. 
Therefore, $i^*\omega\not\in H$ and hence $H$ is a subspace of $\textup{L}_{\mathbb{C}}^1(X)|_E$ of codimension $1$.  

Now, let us consider the following subset of $\textup{L}_{\mathbb{C}}^1(X)|_E$:
$$S:=\{\alpha|_E~|~ \alpha\in \textup{L}_{\mathbb{C}}^1(X),~(\alpha|_E)^{\dim E}=0 \}.$$ 
\textbf{We claim that:  \hypertarget{$(**)$}{$(**)$} $S$ is a hypersurface in $\textup{L}_{\mathbb{C}}^1(X)|_E$ and $H=(\pi^*\textup{L}_{\mathbb{C}}^1(Y))|_E$ is an irreducible component of $S$ (with respect to the  Zariski topology)}; cf.~\cite[Proof of Lemma 6.2]{MZ18}. 
Suppose the claim \hyperlink{$(**)$}{$(**)$} holds for the time being. 
Since $f^{-1}(E)=E$ by our assumption, the operation $(f|_E)^*$ gives an automorphism of the subspace $\textup{L}_{\mathbb{C}}^1(X)|_E$.  
Similar to the pull-back of cycles in the normal projective case, we may assume $f^*E=aE$ for some $a>0$. 
For any $\beta\in \textup{L}_{\mathbb{C}}^1(X)$, by projection formula, 
$$((f|_E)^*(\beta|_E))^{\dim E}=((f^*\beta)|_E)^{\dim E}=(f^*\beta)^{\dim E}\cdot E=\frac{\deg f}{a}\beta^{\dim E}\cdot E=\frac{\deg f}{a}(\beta|_E)^{\dim E}.$$ 
Then $\beta|_E\in S$ if and only if $(f^*\beta)|_E\in S$. 
Therefore, $S$ is $(f|_E)^*$-stable. 
By our claim \hyperlink{$(**)$}{$(**)$}, with $f$ replaced by its power, $H$ is also $(f|_E)^*$-invariant. 
As a result, for any $\beta\in \text{L}_{\mathbb{C}}^1(Y)$,
\begin{align}\label{equa-sym}
(\pi^*\beta)|_E=(f|_E)^*((\pi^*\beta')|_E)=(f^*(\pi^*\beta'))|_E,	
\end{align}
for some $\beta'\in \textup{L}_{\mathbb{C}}^1(Y)$. 
By \cite[Lemma 8.1]{Zho21}, $f^{-1}(R_{\Gamma})=R_{\Gamma'}$ with $\Gamma'\in\overline{\textup{NA}}(X)$ being  a  positive closed $(n-1,n-1)$-current. 
Since $f^{-1}(E)=E$, for each curve $C$ with $[C]\in R_\Gamma$, there exists some curve $C'\subseteq E$ such that $f(C')=C$ and $[C']\in R_{\Gamma'}$.  
Write $f_*[C']=k[C]$ for some $k>0$. 
Then by Equation (\ref{equa-sym}), for every $\beta\in \textup{L}_{\mathbb{C}}^1(Y)$, there exists some $\beta'\in\textup{L}_{\mathbb{C}}^1(Y)$ such that
$$\pi^*\beta\cdot C'=f^*\pi^*\beta'\cdot C'=\pi^*\beta'\cdot kC=0,$$
noting that $\Gamma$ is perpendicular to $\pi^*\textup{L}_{\mathbb{C}}^1(Y)$ by the exact sequence (\ref{exact}). 
Hence, $R_{\Gamma'}=R_{\Gamma}$ by the choice of $\pi$, and thus $f(R_{\Gamma})=R_{\Gamma}$. 
Since $\pi$ is uniquely determined by $R_{\Gamma}$, it follows from the rigidity lemma that $f$ descends to a holomorphic $g$ on $Y$. 

Now, the only thing is to prove our claim \hyperlink{$(**)$}{$(**)$}. 
Let us fix a basis $\{v_1,\ldots, v_k\}$ of $\textup{L}_{\mathbb{C}}^1(X)|_E$. 
Then we have
$$S=\{(x_1,\ldots,x_k)\in\mathbb{C}^k~|~(\sum_{i=1}^k x_iv_i)^{\dim E}=0\}.$$
This implies that  $S$ is determined by a homogeneous polynomial of degree $\dim E$. 
Note that the polynomial has at least one coefficient non-vanishing. 
Indeed, the coefficient of the monomial $\prod_i x_i^{l_i}$ is the intersection number $v_1^{l_1}\cup\cdots\cup v_k^{l_k}$. 
For any K\"ahler class $\omega$ on $X$, the restriction $\omega|_E\in \textup{L}_{\mathbb{C}}^1(X)|_E$ is still K\"ahler (cf.~\cite[Proposition 3.5]{GK20} or \cite[Proposition 2.6]{Zho21}); hence $(\omega|_E)^{\dim E}>0$. 
Since $\omega|_E$ is a linear combination of the basis $\{v_j\}_{j=1}^k$, there exists an item $\prod_i v_i^{l_i}\neq 0$; hence $S$ is determined by a non-zero homogeneous polynomial, which proves the first part of our claim. 
Moreover, for any $(\pi^*\gamma)|_E\in H$ with $\gamma\in\textup{L}_{\mathbb{C}}^1(Y)$, we have
$$\int_E((\pi^*\gamma)|_E)^{\dim E}=\int_X(\pi^*\gamma)^{\dim E}\cdot E=\int_X \gamma^{\dim E}\cup (\pi_* [E])=0,$$
noting that $E$ is $\pi$-exceptional. 
This implies that $H\subseteq S$.
Since $H$ and $S$ have the same codimension in $\textup{L}_{\mathbb{C}}^1(X)|_E$,
the linear subspace $H$ of $S$ is thus an irreducible component of $S$, and the second part of our claim is  proved. 
\end{proof}

\begin{proposition}[{cf.~\cite[Proposition 8.4]{Zho21}}]\label{prodivisorial}
Let $f:X\rightarrow X$ be an int-amplified endomorphism of a normal $\mathbb{Q}$-factorial compact K\"ahler klt threefold $X$. If $\pi:X\rightarrow Y$ is a divisorial contraction of a $K_X$-negative extremal ray $R_\Gamma:=\mathbb{R}_{\ge 0}[\Gamma]$ generated by a positive closed $(2,2)$-current $\Gamma$. Then after iteration, $f$ descends to an int-amplified endomorphism on $Y$.
\end{proposition}
\begin{proof}
The proof is the same as \cite[Proposition 8.4]{Zho21} after replacing 	\cite[Lemma 8.3]{Zho21} by Lemma \ref{lemequivariant}.
\end{proof}

With the same proof as in \cite[Lemma 3.6]{Zha10}, we have the following lemma.
\begin{lemma}[{cf.~\cite[Lemma 8.5]{Zho21}, \cite[Lemma 3.6]{Zha10}}]\label{flipdescendingholomorphic}
Let $f:X\rightarrow X$ be a surjective endomorphism of a normal $\mathbb{Q}$-factorial compact K\"ahler klt threefold $X$  and $\sigma:X\dashrightarrow X'$ a flip with $\pi:X\rightarrow Y$ the corresponding flipping contraction of a $K_X$-negative extremal ray $R_\Gamma:=\mathbb{R}_{\ge 0}[\Gamma]$ generated by a  positive closed $(2,2)$-current $\Gamma$. Suppose that $R_{f_*\Gamma}=R_\Gamma$. Then the dominant meromorphic map $f^+:X^+\dashrightarrow X^+$ induced from $f$, is holomorphic. 
Both $f$ and $f^+$ descend to one and the same endomorphism of $Y$. 
\end{lemma}

\begin{proposition}[{cf.~\cite[Proposition 8.6]{Zho21}, \cite[Lemma 6.5]{MZ18}}]\label{proflipdescending}
Let $f:X\rightarrow X$ be an int-amplified endomorphism of a normal $\mathbb{Q}$-factorial compact K\"ahler klt threefold $X$ and $\sigma:X\dashrightarrow X'$ a flip with $\pi:X\rightarrow Y$ the corresponding flipping contraction of a $K_X$-negative extremal ray $R_\Gamma:=\mathbb{R}_{\ge 0}[\Gamma]$ generated by a  positive closed $(2,2)$-current $\Gamma$. Then there exists a $K_X$-flip of $\pi$, $\pi^+:X^+\rightarrow Y$ such that $\pi=\pi^+\circ\sigma$ and for some $s>0$, the commutativity is $f^s$-equivariant.
\end{proposition}
\begin{proof}
The proof is the same as \cite[Proposition 8.6]{Zho21} after replacing \cite[Lemmas 8.3 and 8.5]{Zho21} by Lemmas \ref{lemequivariant} and \ref{flipdescendingholomorphic}.
\end{proof}

With all the preparations settled, we now can prove Theorem \ref{E-MMP-threefolds}.

\begin{proof}[Proof of Theorem \ref{E-MMP-threefolds}]
The proof is the same as \cite[Theorem 1.5]{Zho21} after replacing \cite[Theorem 1.1]{HP15} by the pure case of \cite[Theorem 1.2]{DH22}, and replacing \cite[Theorem 8.7]{Zho21} by Lemmas \ref {lemequivariant}, \ref{flipdescendingholomorphic} and Propositions \ref{prodivisorial}, \ref{proflipdescending}.
\end{proof}

\subsection{Proof of Corollary \ref{simply-connected}}
In the subsection, we prove Corollary \ref{simply-connected} by running the E-MMP.
Note that, for a $Q$-complex torus $X$ admitting an int-amplified endomorphism $f$, it has been shown that there is no proper $f^{-1}$-periodic analytic subvariety in $X$ (cf.~\cite[Lemma 5.4]{Zho21}). 
We shall heavily use this fact to show that the end product of the E-MMP for such $X$ in Corollary \ref{simply-connected} turns out to be a single point. 
\begin{proof}[Proof of Corollary \ref{simply-connected}]
First, we show  the canonical divisor $K_X$ is not pseudo-effective.
Suppose the contrary that   $K_X$ is pseudo-effective. 
Then $X$ is a $Q$-complex torus by Theorem \ref{thm-Q-torus}. 
Since $X$ is simply connected, the irregularity $q(X)=0$. 
Since there exists a complex $3$-torus $T$ such that $T\rightarrow X$ is  quasi-\'etale and thus \'etale by the purity of branch theorem (cf.~\cite{GR55}),  
our $T\rightarrow X$ is trivial and $X$ itself is a complex torus.
However, this contradicts $q(X)=0$. 
Therefore, $K_X$ is not pseudo-effective.

Second, we show that $X$ is projective.
Suppose the contrary that $X$ is non-projective. 
By \cite[Theorem 1.1]{HP15} and Theorem \ref{E-MMP-threefolds}, with $f$ replaced by its power, there exists an $f$-equivariant minimal model program from $X$ to a K\"ahler surface $S$ with only klt singularities. 
Since $X$ is non-algebraic, our $S$ is non-uniruled; for otherwise, the MRC fibration of $X$ will have the target of dimension $\le 1$ (and hence $X$ is projective; see~\cite[Introduction]{HP15}). 
Therefore, $K_S$ is pseudo-effective. 
Then $S$ is a $2$-dimensional $Q$-complex torus (cf.~\cite[Theorem 1.4]{Zho21}), i.e., there is a finite morphism $T\to S$ from a complex torus $T$, which is \'etale in codimension one. 
Since $S$ does not contain any totally invariant periodic proper subsets (cf.~\cite[Lemma 5.4]{Zho21}), there does not exist any flips or divisorial contractions when running the E-MMP for $X$, noting that $\dim X=3$. 
Hence, with $f$ replaced by its power, $\pi:X\rightarrow S$ is an  $f$-equivariant Fano contraction of a $K_X$-negative extremal ray. 
Let $\widetilde{X}$ be the normalization of the main component of $X\times_S T$. 
Then we have the following commutative diagram:
$$\xymatrix{\widetilde{X}\ar[drrr]^{p_2\circ n}\ar[dr]^n\ar[ddr]_{\nu_X}&&&\\
&X\times_ST\ar[rr]_{p_2}\ar[d]^{p_1}&&T\ar[d]^{\nu_S}\\
&X\ar[rr]_\pi&&S
}$$
where $n$ is the normalization and $p_1$ and $p_2$ are natural projections. \textbf{We claim that $\nu_X$ is \'etale.} 
Suppose the claim for the time being. 
Then  $1=\deg \nu_X=\deg \nu_S$, noting that $X$ is simply connected. 
So $S=T$ is a complex $2$-torus and $0<q(S)\le q(X)=0$, a contradiction. 
Hence, $X$ is projective. 

Now we are left  to prove our claim. Since $\nu_S$ is \'etale in codimension one, the  branch locus of $\nu_S$ consists of finitely many points. Note that \'etaleness is stable under base change and  $\pi$ is equidimensional. Therefore, $p_1$ is \'etale outside finitely many curves, which is a codimension $\ge 2$ subset. 
Since $X$ is smooth, around the \'etale locus of $p_1$, $X\times_S T$ is smooth and hence normal. 
So $n$ is biholomorphic around the \'etale locus of $X\times_S T$. This implies that $\nu_X$ is \'etale outside a codimension $\ge 2$ subset. By the purity of branch theorem (cf.~\cite{GR55}), $\nu_X$ is \'etale. 
This proves our claim.

Finally, we are now in the case when $X$ is projective and $K_X$ is not pseudo-effective.
By the same proof as in the second step, we know that the end product $Z$ of the E-MMP has dimension $\le 1$.
If $Z$ is a single point, then it follows from Theorem \ref{E-MMP-threefolds} or \cite[Theorem 1.10]{Men20} that $X$ is rationally connected. 
If $\dim Z=1$, then $q(Z)\le q(X)=0$ which implies that $Z$ is rational and hence $X$ is also  rationally connected; see \cite[Corollary 1.3]{GHS03}.
Thus, $X$ is (projective and) rationally connected. 
Applying \cite[Corollary 1.4]{Yos21}, we see that $X$ is of Fano type and our theorem is thus proved.
\end{proof}

\end{document}